\documentclass[twoside,leqno,12pt]{article}

\usepackage{amssymb}
\usepackage{amsmath}
\usepackage{amscd}
\usepackage{amsthm}
\usepackage{ifthen}
\usepackage{calc}
\usepackage{xcolor}
\usepackage{graphics}
\usepackage{arydshln}

\definecolor{darkgreen}{rgb}{0,0.5,0}

\def\qed{\hfill$\square$}

\numberwithin{equation}{section}
\newtheorem{thm}[equation]{\sc Theorem}
\newtheorem{lem}[equation]{\sc Lemma}
\newtheorem{cor}[equation]{\sc Corollary}
\newtheorem{prop}[equation]{\sc Proposition}
\newtheorem{conj}[equation]{\sc Conjecture}

\newtheoremstyle{notation}{3pt}{3pt}{}{}{\itshape}{:}{.5em}{\thmname{#1}}
\theoremstyle{notation}
\newtheorem{notation}{\it Notation}

\newtheorem{rem}{\it Remark}
\newtheorem{defin}{\it Definition}
\newtheorem{ex}{\it Example}

\makeatletter
\renewcommand{\@seccntformat }[1]{\csname the#1\endcsname. }
\makeatother
\newcommand{\sort}{{\sf sort}}

\def\longarr#1#2{{\buildrel{#1} \over {\hbox to #2pt{\rightarrowfill}}}}

\def\Hom{\mbox{\rm Hom}}

\def\vedge#1{{\buildrel{#1} \over {\hbox to
20pt{\hspace{-0.2em}$-$\hspace{-0.2em}$-$\hspace{-0.2em}$-$ }}}}

\input prepictex   \input pictex    
\input postpictex

\newcounter{boxsize}
\newcounter{tempcounter}
\newcommand{\smallentryformat}{\scriptstyle\sf}

\newcommand\smbox{\put(0,0){\line(1,0){\value{boxsize}}}%
  \put(\value{boxsize},0){\line(0,1){\value{boxsize}}}%
  \put(0,0){\line(0,1){\value{boxsize}}}%
  \put(0,\value{boxsize}){\line(1,0){\value{boxsize}}}}
\newcommand\alphambox[1]{\put(0,0)\smbox%
  \put(0,0){\makebox(\value{boxsize},\value{boxsize})[c]{%
      $\smallentryformat#1$}}}
\newcommand\singlebox[1]{\raisebox{-.4ex}{\begin{picture}(4,0)\setcounter{boxsize}{3}%
    \put(0,0)\smbox%
    \put(0,0){\makebox(\value{boxsize},\value{boxsize})[c]{%
      $\scriptstyle\sf#1$}}\end{picture}}}

\def\alphatab#1#2#3#4#5#6{\begin{picture}(5,8)(0,3)
    \put(0,7){\alphambox{#1}}
    \put(3,7){\alphambox{#2}}
    \put(6,7){\alphambox{#3}}
    \put(0,4){\alphambox{#4}}
    \put(3,4){\alphambox{#5}}
    \put(0,1){\alphambox{#6}}
  \end{picture}}
\def\arr#1#2{\arrow <2mm> [0.25,0.75] from #1 to #2}

\newcommand\boxes[2]{\ifthenelse{#2=3}{$\scriptstyle P_2^{#1}$}{%
                                       $\scriptstyle P_{#2}^{#1}$}}

\begin{document}
\thispagestyle{empty}
\color{black}
\phantom m\vspace{-2cm}
%\noindent{\footnotesize[{\tt \ver,} \today]}

\bigskip\bigskip
\begin{center}
{\large\bf Two Partial Orders for {Standard Young Tableaux}} 
\end{center}

\smallskip

\begin{center}
Justyna Kosakowska, Markus Schmidmeier and Hugh Thomas
\footnote{The first named author is partially supported 
by Research Grant No.\ DEC-2011/ 02/A/ ST1/00216
of the Polish National Science Center.}
\footnote{This research is partially supported 
by a~Travel and Collaboration Grant from the Simons Foundation (Grant number
245848 to the second named author).}
\footnote{The third named author was partially supported by an
NSERC Discovery Grant and the Canada Research Chairs program.}
\vspace{1cm}

\bigskip \parbox{10cm}{\footnotesize{\bf Abstract:}
In this manuscript we show that two partial orders defined on the set
of standard Young tableaux  of shape $\alpha$ are equivalent.
In fact, we give two proofs for the equivalence of 
the box order and the dominance order for {tableaux}. Both are algorithmic.
The first of these proofs emphasizes links to the Bruhat order for the symmetric group
and the second provides a
more straightforward construction of the cover relations.
This work is motivated by the known result that the equivalence of the 
two combinatorial orders leads to a description of the geometry of 
the representation space of invariant subspaces of nilpotent linear
operators.
}

\medskip \parbox{10cm}{\footnotesize{\bf MSC 2010:} 
Primary: 05E10, %combinatorial aspects of representation theory
Secondary:
47A15  % (invariant subspaces)
}

\medskip \parbox{10cm}{\footnotesize{\bf Key words:} 
Littlewood-Richardson filling, 
Bruhat order, nilpotent linear operators}
\end{center}

%===================================================================
\section{Introduction}
%===================================================================

Let $\alpha$ be a~partition.
By $\mathcal T_\alpha$
we denote the set of all standard Young tableaux  (SYT) of
shape $\alpha$.
In Section~\ref{section-def} we define two partial orders 
$\leq_{\sf box}$ and $\leq_{\sf dom}$ on the set $\mathcal T_\alpha$. 
These orders are defined combinatorially and are of importance in the 
theory of invariant subspaces of nilpotent linear operators. They control 
the geometry  of varieties of
invariant subspaces of nilpotent linear operators, 
as they describe the degeneration relation and the boundaries of the
irreducible components,
see \cite{ks-hall,ks-survey,ks-strips} and Section \ref{sec-motivation}. Therefore,
it is important to investigate properties of these orders. One of the main results
of  the paper is the following theorem.

\begin{thm}
  Let $X$, $Z$ be SYTs
  of the same shape.
  The following conditions are
  equivalent.
  \begin{enumerate}
  \item  $Z\leq_{\sf box}X$,
  \item  $Z\leq_{\sf dom}X$.
  \end{enumerate}
  \label{thm-main}
\end{thm}

We show in \cite{ks-strips} that several other relations of geometric
or of algebraic nature lie between the box and the dominance relations.
If those two are equal, then all the relations coincide.

 We remark that partial orders on standard Young tableaux have
  been considered elsewhere, such as, notably, 
  by Ta\c{s}kin in \cite{tas}. Note, though, that the posets defined there are
  defined on all SYTs of size
  $n$ at once.  In fact, for three of the partial orders considered there
  (the weak, KL, and
  geometric orders), two non-equal fillings of the same shape are
  always incomparable \cite[Proposition 3.5]{tas}, so the restriction to
  SYT of fixed shape is not interesting.  The order which Ta\c{s}kin refers
  to as chain order is similar to, but strictly weaker than, our
  dominance order (after transposition).

We present two different proofs of Theorem~\ref{thm-main}. 
Both proofs are constructive. 
The first one, presented in Section \ref{section-first-proof},
shows connections of our problem with the Bruhat order in the symmetric group $S_n$
where $n=|\alpha|$.

The second proof, given in Section 
\ref{section-algorithm}, gives a~more
straightforward algorithm 
that applies operations on entries of SYTs. Given two SYTs such that 
$Z\leq_{\sf dom} X$, both 
algorithms compute a~sequence
of box moves that convert $X$ to $Z$. This proves that
$Z\leq_{\sf box} X$.

%\begin{red}Both algorithms depend on choices.  We show in Section~\ref{sec-comparison} 
%that for suitable choices, the two algorithms produce the same sequence of box moves.
%\end{red}

Finally,
in Section \ref{sec-prop}, we describe some properties of the poset 
$(\mathcal T_\alpha,\leq_{\sf box})$.
We prove that 
there exists exactly one minimal and exactly one maximal element, 
and that all saturated chains have the same length.

%===================================================================
\section{Definitions and notation}
\label{section-def}
%===================================================================

\medskip

Following \cite{fult,macd} we recall definitions and notation connected with SYTs.

\begin{notation}
Recall that a~{\it partition} $\alpha=(\alpha_1,\ldots, \alpha_s)$
is a~finite non-increasing sequence of natural numbers;
we picture $\alpha$ by its Young diagram which consists of $s$ rows of
lengths given by the parts of $\alpha$. {Let $|\alpha|=\alpha_1+\ldots+\alpha_s$.}
 The {\it transpose} $\alpha'$ of $\alpha$ is given by the formula
 $$\alpha'_j=\#\{i:\alpha_i\geq j\},$$
 it is pictured by the transpose of the Young diagram for $\alpha$.
Two partitions $\alpha$, $\widetilde{\alpha}$ are in the {\it dominance partial order,}
  in symbols $\alpha\leq_{\sf dom}\widetilde{\alpha}$, if the  inequality 
  $$\alpha_1+\cdots+\alpha_j\leq\widetilde{\alpha}_1+\cdots+\widetilde{\alpha}_j$$
  holds for each $j$.
\end{notation}
\smallskip

 \begin{notation} Let $\alpha$ be a~partition.
 {\it A~standard Young tableau (SYT)} of shape $\alpha$ is a~tableau of shape $\alpha$
 filled by the numbers $1,2,\ldots,|\alpha|$
 such that its entries increase along rows and down columns. Let $\mathcal{T}_\alpha$ be the set of all 
 SYTs of shape $\alpha$.
 
Given a~SYT $X$ we number its rows starting from the top and going down.
\end{notation}

\begin{ex}
Let $\alpha=(2,2,1)$. There are exactly five SYTs of shape $\alpha$:
$$ \begin{picture}(5,8)(0,3)
%\multiput(0,12)(3,0)4{\smbox}
\put(0,7){\alphambox{1}}
\put(3,7){\alphambox{2}}
\put(0,4){\alphambox{3}}
\put(3,4){\alphambox{4}}
\put(0,1){\alphambox{5}}
%\multiput(0,9)(3,0)3{\smbox}
\end{picture} 
\;\;\;\;\;\;\;\;
\begin{picture}(5,8)(0,3)
%\multiput(0,12)(3,0)4{\smbox}
\put(0,7){\alphambox{1}}
\put(3,7){\alphambox{2}}
\put(0,4){\alphambox{3}}
\put(3,4){\alphambox{5}}
\put(0,1){\alphambox{4}}
%\multiput(0,9)(3,0)3{\smbox}
\end{picture}
\;\;\;\;\;\;\;\;
\begin{picture}(5,8)(0,3)
%\multiput(0,12)(3,0)4{\smbox}
\put(0,7){\alphambox{1}}
\put(3,7){\alphambox{3}}
\put(0,4){\alphambox{2}}
\put(3,4){\alphambox{4}}
\put(0,1){\alphambox{5}}
%\multiput(0,9)(3,0)3{\smbox}
\end{picture}
\;\;\;\;\;\;\;\;
\begin{picture}(5,8)(0,3)
%\multiput(0,12)(3,0)4{\smbox}
\put(0,7){\alphambox{1}}
\put(3,7){\alphambox{3}}
\put(0,4){\alphambox{2}}
\put(3,4){\alphambox{5}}
\put(0,1){\alphambox{4}}
%\multiput(0,9)(3,0)3{\smbox}
\end{picture}
\;\;\;\;\;\;\;\;
\begin{picture}(5,8)(0,3)
%\multiput(0,12)(3,0)4{\smbox}
\put(0,7){\alphambox{1}}
\put(3,7){\alphambox{4}}
\put(0,4){\alphambox{2}}
\put(3,4){\alphambox{5}}
\put(0,1){\alphambox{3}}
%\multiput(0,9)(3,0)3{\smbox}
\end{picture}
 $$ 
In the first tableau the entries $1$ and $2$ are in the row number $1$; the entries $3$, $4$  in the row number $2$, etc. 
\end{ex}

\begin{notation}
One can represent a~SYT $X$ by a~sequence of partitions

 $$X=[\gamma^{(1)},\ldots,\gamma^{(|\alpha|)}]$$
% (important to have this %, so the paragraph continues with the next line)
where 
$\gamma^{(i)}$ denotes the region in the Young diagram  {$\alpha$}
which contains the entries
$\singlebox{1}$, $\ldots$, $\singlebox{i}$.
\end{notation}

\smallskip
In the example above, the first filling is given by the sequence of partitions

$X=[(1),(2),(2,1),(2,2),(2,2,1)]$.

\medskip
We introduce two partial orders on the set 
{$\mathcal T_\alpha$ of all SYTs of 
the same shape}. 

\begin{defin}
Two 
SYTs $Z=[\delta^{(1)},\ldots,\delta^{(|\alpha|)}]$, 
$X=[\gamma^{(1)},\ldots,\gamma^{(|\alpha|)}]$ of the same shape are
{\it related in the dominance partial order,}
in symbols $Z\leq_{\sf dom} X$,
if for each $i$, the corresponding partitions 
$\delta^{(i)}$, $\gamma^{(i)}$ are related in the dominance partial order, i.e. 
$\delta^{(i)}\leq_{\sf dom}\gamma^{(i)}$.
\end{defin}

In the example above of SYTs of shape $(2,2,1)$, the first one
  listed is the largest in dominance order, and the last one is the
  smallest,  while, for example, the second and  third are incomparable
    in dominance order.

\begin{defin}
Suppose $X$, $Z$ are SYTs of the same shape.
We say $Z$ is obtained from $X$ by {\it a~decreasing box move} if
$Z$ is obtained by swapping two entries in $X$ so that the smaller entry
  winds up in  the higher-numbered row.  
We denote by $\leq_{\sf box}$ the partial order generated by box moves.
\end{defin}

Here is an~example:
$$
Z:\quad\begin{picture}(5,8)(0,3)
%\multiput(0,12)(3,0)4{\smbox}
\put(0,7){\alphambox{1}}
\put(3,7){\alphambox{4}}
\put(0,4){\alphambox{2}}
\put(3,4){\alphambox{5}}
\put(0,1){\alphambox{3}}
%\multiput(0,9)(3,0)3{\smbox}
\end{picture} 
\;\;\;\;\;\;\;\;
<_{\sf box}
\;\;\;\;\;\;\;\;
X:\quad\begin{picture}(5,8)(0,3)
%\multiput(0,12)(3,0)4{\smbox}
\put(0,7){\alphambox{1}}
\put(3,7){\alphambox{2}}
\put(0,4){\alphambox{3}}
\put(3,4){\alphambox{4}}
\put(0,1){\alphambox{5}}
%\multiput(0,9)(3,0)3{\smbox}
\end{picture}
$$
To get $Z$ we apply to $X$ the following sequence of moves.
First we exchange $\singlebox 2$ and $\singlebox 3$, then
$\singlebox 4$ and $\singlebox5$
and finally $\singlebox 3$ and $\singlebox 4$.
\medskip

We finish this section by establishing the following fact.

\begin{lem}\label{lem-box-to-dom}
For SYTs of the same shape, the $\leq_{\sf box}$-order 
implies the $\leq_{\sf dom}$- order.
\end{lem}

\begin{proof}
Suppose that the {SYT}
{$Z=[\delta^{(1)},\ldots,\delta^{(|\alpha|)}]$}
is obtained from 
{$X=\allowbreak[\gamma^{(1)},\ldots,\gamma^{(|\alpha|)}]$}
by a~decreasing box move based on swapping 
entries $a$ and $b$ with, say, $a<b$.
{Note that} the partitions $\gamma^{(1)},\ldots,\gamma^{(a-1)}$,
and $\gamma^{(b)},\ldots,\gamma^{(|\alpha|)}$ remain unchanged.
The partitions $\delta^{(\ell)}$, $\gamma^{(\ell)}$ for $a\leq \ell<b$ 
are different and satisfy $\delta^{(\ell)}<_{\sf dom}\gamma^{(\ell)}$ 
(since the defining partial sums can only  decrease).  
This shows that $Z<_{\sf dom} X$.
\end{proof}

\section{{LR-fillings and} motivation}
%===================
\label{sec-motivation}

  Our investigation of partial orders for  SYTs is motivated by
  its connections with {LR-fillings} and an application to short exact sequences of nilpotent linear operators.

  Fix two partitions $\gamma\subseteq\beta$
such that the Young diagram for $\gamma$ is contained
in the Young diagram for $\beta$. 
The skew diagram  $\beta\setminus\gamma$ 
is said to be a~{\it vertical strip} if $\beta_i\leq \gamma_i+1$ 
holds for all $i$, and a~{\it horizontal strip} if $\beta'\setminus\gamma'$
is a~vertical strip. A~{\it rook strip} is a horizontal and vertical strip.

Given three partitions $\alpha=(\alpha_1,\ldots, \alpha_s)$, $\beta$, $\gamma$,
we will consider fillings of $\beta\setminus\gamma$ which have
$\alpha_1$ entries $1$'s, $\alpha_2$ entries $2$'s, etc.
We describe such a~filling as having the {\it content} $\alpha$
and the {\it shape} $\beta\setminus\gamma$.
  The {\it type} of the filling,
  $(\alpha,\beta,\gamma)$, records the content and shape.
A~filling is said to be
an {\it LR-filling} if the following three conditions are satisfied:
\begin{itemize}
\item in each row, the entries are weakly increasing,
\item in each column, the entries are strictly increasing,
\item ({\it lattice permutation property}) for each $u>1$ and for each column $c$:
  on the right hand side of $c$, the number of entries $u-1$
  is at least the number of entries $u$.
\end{itemize}
    
      The number of LR-fillings of type $(\alpha,\beta,\gamma)$ is the Littlewood-Richardson coefficient $c_{\alpha,\gamma}^\beta$ which plays an important r\^ole
      in symmetric functions, Schubert calculus, and the representation theory of the general linear group,  see \cite{fult,macd}. 

\smallskip

\begin{ex}  Let $\alpha=(2,2,1)$, $\beta=(4,3,3,2,1)$, $\gamma=(3,2,2,1)$.
We have to fill the skew
diagram $\beta\setminus\gamma$ with two $\singlebox1$'s, two $\singlebox2$'s, 
and one $\singlebox3$.
Due to the conditions on an LR-filling, this can be done in exactly three ways.

$$ \setcounter{boxsize}{3}
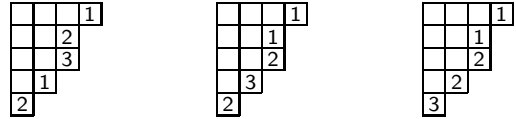
\begin{picture}(18,12)(0,2)
\multiput(0,12)(3,0)4{\smbox}
\put(0,0){\alphambox{2}}
\multiput(0,9)(3,0)3{\smbox}
\put(3,3){\alphambox{1}}
\put(6,6){\alphambox{3}}
\multiput(0,6)(3,0)3{\smbox}
\multiput(0,3)(3,0)2{\smbox}
\put(6,9){\alphambox{2}}
\put(9,12){\alphambox{1}}
\end{picture}
\;\;\;\;\;\;\;\;
\setcounter{boxsize}{3}
\begin{picture}(18,12)(0,2)
\multiput(0,12)(3,0)4{\smbox}
\put(0,0){\alphambox{2}}
\multiput(0,9)(3,0)3{\smbox}
\put(3,3){\alphambox{3}}
\put(6,6){\alphambox{2}}
\multiput(0,6)(3,0)3{\smbox}
\multiput(0,3)(3,0)2{\smbox}
\put(6,9){\alphambox{1}}
\put(9,12){\alphambox{1}}
\end{picture}
\;\;\;\;\;\;\;\;
\setcounter{boxsize}{3}
\begin{picture}(18,12)(0,2)
\multiput(0,12)(3,0)4{\smbox}
\put(0,0){\alphambox{3}}
\multiput(0,9)(3,0)3{\smbox}
\put(3,3){\alphambox{2}}
\put(6,6){\alphambox{2}}
\multiput(0,6)(3,0)3{\smbox}
\multiput(0,3)(3,0)2{\smbox}
\put(6,9){\alphambox{1}}
\put(9,12){\alphambox{1}}
\end{picture}
$$
In this example, $\beta\setminus\gamma$ is~a~vertical
but not a~horizontal strip.
\end{ex}

 %\begin{rem}
 %\begin{enumerate}
%\item
  Let $\beta\setminus \gamma$ be a~rook strip. Note that there is a~bijection
   between the set $\mathcal{T}_{\alpha,\gamma}^\beta$ of LR-fillings of type $(\alpha,\beta,\gamma)$
   and the set $\mathcal{T}_\alpha$ of the SYTs of shape $\alpha$. A~bijection
   $\Phi^\beta_\gamma:\mathcal{T}_{\alpha,\gamma}^\beta\to \mathcal{T}_\alpha$ is constructed as follows.
   For an LR-filling $X$ we denote by $\tau(X)=(\tau_1,\ldots,\tau_{|\alpha|})$ the
 list of entries when
 reading columns from the top down, starting with the rightmost
  column and moving left. Fix $i\leq s$,
  where $s$ is
  the number of rows of $\alpha$. Let $j_1<j_2<\ldots<j_{n_i}$ be all elements $j$ such that
  $\tau_j=i$.  We write the 
  elements $j_1,j_2,\ldots,j_{n_i}$
  in the $i$-th row of the corresponding SYT of shape $\alpha$.

  Box order and dominance order were defined on LR-fillings in \cite{ks-strips}.
  We will review these definitions, and verify that they agree under the
  bijection defined above with the orders we have already defined on
  standard tableaux.
  
  %\item
   The box order is defined on LR-fillings of type $(\alpha,\beta,\gamma)$
    by saying that the filling $Z$ is obtained by a {\it decreasing
    box move} from the filling
    $X$ if $Z$ is obtained from $X$ by swapping two entries, such that the
    smaller entry winds up in the lower position.

    \begin{lem} For LR-fillings $X,Z$ of type $(\alpha,\beta,\gamma)$
      we have that $\Phi^\beta_\gamma(X)\geq_{\sf box} \Phi^\beta_\gamma({Z})$ if and
      only if $X\geq_{\sf box} Z$. \end{lem}

    \begin{proof} 
      
      Let $X,Z$ be LR-fillings of type $(\alpha,\beta,\gamma)$.
      Suppose we obtain $Z$ from $X$ by swapping the positions of entries
      $\tau_i$ and $\tau_j$ with $i<j$, and in $X$ we have
      $\tau_i=a, \tau_j=b$, with $a<b$.

      Now consider $\Phi^\beta_\gamma(X)$ and $\Phi^\beta_\gamma(Z)$. In
      $\Phi^\beta_\gamma(X)$, we have $i$ in row $a$ and $j$ in row $b$.
      By the definition of $\Phi^\beta_\gamma$, we see that
      we obtain $\Phi^\beta_\gamma(Z)$ from $\Phi^\beta_\gamma(X)$
      by swapping the entries $i$ and $j$
      and resorting the two rows if necessary so that they are increasing.
      This resorting step is not allowed in our
      definition of decreasing box moves for SYTs.

      However, it turns out that there is a sequence of legal
      decreasing box moves 
      which suffice to transform $\Phi^\beta_\gamma(X)$ into $\Phi^\beta_\gamma(Z)$.

      Let $I$ be the entries of row $a$ of $\Phi^\beta_\gamma(X)$ which are weakly
      between $i$ and $j$, and similarly let $J$ be the entries of row $b$
      of $\Phi^\beta_\gamma(X)$ which are weakly between $i$ and $j$.
      The proof
      is by induction on $|I|$.

      The base case is when $|I|=1$, in which case $I=\{i\}$. In this case,
      if $J=\{j_r<j_{r-1}<\dots<j_1=j$, perform the following sequence of
      swaps:
      $$(i,j_r),(j_r,j_{r-1}), \dots (j_2,j_1)$$
      This has the effect of swapping $i$ and $j$ and resorting row $b$.

      Now suppose that $|I|=p>1$. Let the largest element of $I$ be $i_p$,
      and let $J'=\{j_r<\dots<j_2<j_1=j\}$
      be the elements of row $b$ which are weakly between
      $i_p$ and $j$. Perform the same sequence of swaps as before using
      $i_p$ and $J'$ in place of $i$ and $J$.  This has the effect of swapping $i_p$ and $j$, and
      resorting row $b$ if necessary. To complete the desired effect, it suffices to swap $i$ and $i_p$ (which is now in row $b$). This is another instance of the same problem,
      but with the original $i_p$ having been removed from $I$. We are therefore done by induction.

      The fact that all the tableaux we pass through are indeed standard,
      follows from the fact that the starting and ending tableaux are
      standard by hypothesis.

      The converse direction is obvious: if there is a decreasing box move
      from $\Phi^\beta_\gamma(X)$ to $\Phi^\beta_\gamma(Z)$, then the corresponding
      move from $X$ to $Z$ is a decreasing box move.
    \end{proof}

     For an example of the phenomenon considered in the previous proof, consider the following LR-fillings, which are related by the box move swapping the entries in rows 3 and 6.
$$\raisebox{5mm}{$Z:$} \quad
\begin{picture}(18,15)(0,3)
\multiput(0,15)(3,0)5{\smbox}
\put(15,15){\alphambox{1}}
\multiput(0,12)(3,0)4{\smbox}
\put(12,12){\alphambox{1}}
\multiput(0,9)(3,0)3{\smbox}
\put(9,9){\alphambox{2}}
\multiput(0,6)(3,0)2{\smbox}
\put(6,6){\alphambox{2}}
\put(0,3){\smbox}
\put(3,3){\alphambox{3}}
\put(0,0){\alphambox{1}}
\end{picture}
\quad\raisebox{5mm}{$<_{\sf box}$} \;\;\;
\raisebox{5mm}{$X:$} \quad
\begin{picture}(18,15)(0,3)
\multiput(0,15)(3,0)5{\smbox}
\put(15,15){\alphambox{1}}
\multiput(0,12)(3,0)4{\smbox}
\put(12,12){\alphambox{1}}
\multiput(0,9)(3,0)3{\smbox}
\put(9,9){\alphambox{1}}
\multiput(0,6)(3,0)2{\smbox}
\put(6,6){\alphambox{2}}
\put(0,3){\smbox}
\put(3,3){\alphambox{3}}
\put(0,0){\alphambox{2}}
\end{picture}
$$

The corresponding tableaux are

$$
 \raisebox{1.5mm}{$\Phi^\beta_\gamma(Z):$} \quad \alphatab 126345 \;\;\;\;\quad  \quad \;\;
  \raisebox{1.5mm}{$\Phi^\beta_\gamma(X):$} \quad \alphatab 123465
  $$

  Swapping 3 and 6 is not a legal decreasing box move starting from
  $\Phi^\beta_\gamma(X)$, but we can
  first swap 3 and 4, then 4 and 6.

 Similarly, one can define {\it dominance order} on the set of
  LR-fillings.
  One can represent an LR-filling $X$ by a~sequence of partitions
$$X=[\gamma^{(0)},\ldots,\gamma^{(s)}]$$ 
where $s$ is the number of rows of $\alpha$ and
$\gamma^{(i)}$ denotes the region in the Young diagram $\beta$
which contains the entries
$\singlebox{}$, $\singlebox{1}$, $\ldots$, $\singlebox{i}$.
 If  $X$ has type $(\alpha,\beta,\gamma)$, then $\gamma=\gamma^{(0)}$,
$\beta=\gamma^{(s)}$, and $\alpha_i=|\gamma^{(i)}\setminus\gamma^{(i-1)}|$ 
 for $i=1,\ldots,s$.  For LR-fillings, we define $X\geq_{\sf dom}Z$
   to mean that for all $1\leq i \leq s$, the $i$-th partition corresponding
   to $X$ is greater than or equal to the $i$-th partition corresponding to $Z$
 in dominance order.

 We then have the following lemma:

  \begin{lem} For LR-fillings $X,Z$ of type $(\alpha,\beta,\gamma)$
    we have $\Phi^\beta_\gamma(X)\geq_{\sf dom} \Phi^\beta_\gamma(Z)$ if and only
    if $X\geq_{\sf dom} Z$. \end{lem}

  \begin{proof} Suppose that $\alpha$ has $s$ rows.  For $1\leq r \leq s$,
    write $\Phi^\beta_\gamma(X)|_{\leq r}$ for the first $r$ rows of $\Phi^\beta_\gamma(X)$.  For $x$ a word, write $\sort (x)$ for the result of sorting the letters
    of $x$ into increasing order.  We can similarly apply $\sort$ to a tableau; the result is again just the word obtained by sorting the entries of the tableau into increasing order.  
    
    $X\geq_{\sf dom} Z$ is equivalent to saying that, for each $1\leq r\leq s$,
    we have
    $\sort(\Phi^\beta_\gamma(X)|_{\leq r}) \leq \sort(\Phi^\beta_\gamma(Z)|_{\leq r})$,
    where the comparison is done coordinatewise.

   Let $\delta^{(i)}$ be the sequence of partitions corresponding to
    $\Phi^\beta_\gamma(X)$, and let $\eta^{(i)}$ be the sequence of partitions
    corresponding to $\Phi^\beta_\gamma(Z)$.  
    
    The condition that, for a fixed $r$, we have $\sort(\Phi^\beta_\gamma(X)|_{\leq r}) \leq \sort(\Phi^\beta_\gamma(Z)|_{\leq r})$, is equivalent to saying that, for
    all $i$, the sum of the first $r$ parts of $\delta^{(i)}$ is greater than
    or equal to the sum of the first $r$ parts of $\eta^{(i)}$.  This condition
    for all $r$ is exactly the definition of $\Phi^\beta_\gamma(X) \geq_{\sf dom} \Phi^\beta_\gamma(Z)$. 
\end{proof} 
  
  %\end{enumerate}
 %\end{rem}

  Let $k$ be an algebraically closed field.
  A nilpotent $k$-linear operator $N=(V,T)$ consists of
  a~finite dimensional $k$-vector space
  $V$ together with a nilpotent $k$-linear map $T:V\to V$.
  Such an operator $N_\alpha=(V,T)$ is given uniquely,
  up to isomorphy, by a partition $\alpha$
  recording the sizes of the Jordan blocks of the action of $T$ on
  the vector space $V$.
  Given two nilpotent linear operators $N=(V,T)$ and $N'=(V',T')$, 
  a morphism from $N$ to $N'$ is a $k$-linear map
  $\phi:V\rightarrow V'$ such that $T'\phi=\phi T$.

  The Green-Klein Theorem \cite{gk} establishes
  the link with LR-fillings:\medskip

  \begin{thm}
    For partitions $\alpha,\beta,\gamma$, there exists a short exact sequence
    $0\to N_\alpha\to N_\beta\to N_\gamma\to 0$ of nilpotent linear operators
    if and only if there is an LR-filling of
    type $(\alpha',\beta',\gamma')$. \qed
  \end{thm}
  
  More precisely, if $A$ is the image of the embedding
  $N_\alpha\to N_\beta=B$ in the short exact sequence, then the
  tableau $X=[\gamma^{(0)},\ldots,\gamma^{(s)}]$ of the sequence is given by
  $s=\min\{\ell:T^\ell A=0\}$,
  and the transposes of the partitions $\gamma^{(\ell)}$
  are given by the Jordan types of the action of $T$ on the factors
  $B/T^\ell A\cong N_{(\gamma^{(\ell)})'}$.

  \medskip
  Together, the $k[T]$-monomorphisms in the short exact sequences 
  form a constructible subset $\mathbb V_{\alpha,\gamma}^{\beta}$
  of the affine variety $\Hom_k(N_\alpha,N_\beta)$.
  Note that each irreducible component $\overline{\mathbb V}_X$
  of $\mathbb V_{\alpha,\gamma}^{\beta}$
  is given as the closure of the set of
  sequences with corresponding LR-filling $X$.  All irreducible components
  have the same dimension.

  \begin{defin}
    Two LR-tableaux $X$, $Z$ of the same type are said to be in
    {\it boundary relation}, $Z\preccurlyeq_{\sf boundary}X$, 
    if $\mathbb V_X\cap \overline{\mathbb V}_Z\neq \emptyset$ holds.
  \end{defin}
  
  \medskip
  The following theorem is shown in \cite{ks-strips}:

  \begin{thm}\label{thm-box-bound}
    Suppose $X$, $Y$ are LR-tableaux of the same type
    and of shape which is a rook strip.
    If $Y$ is obtained from $X$ by a decreasing box move,
    then $Y\prec_{\sf boundary} X$.  \qed
  \end{thm}
  
  More precisely, given $X$, $Y$ in box relation,
  we construct a one-parameter family of
  embeddings $M(\lambda)$, and for each embedding a short exact sequence
  $0\to L\to M(\lambda)\to N\to 0$, such that the following
  properties are satisfied:
  \begin{enumerate}\item $L\oplus N$ has tableau $X$;
  \item the sequence is split exact if $\lambda=0$;
  \item $M(\lambda)$ has tableau $Y$ if $\lambda\neq 0$.
  \end{enumerate}

  Thus, the above result provides a link between the combinatorial
  relation given by box moves, the algebraic relation given by
  short exact sequences and the geometric boundary relation.
  
  \medskip
  In general, the boundary relation implies the dominance relation
  \cite{ks-strips}.
  Hence, the transitive closure $\leq_{\sf boundary}$ of the boundary
  relation $\preccurlyeq_{\sf boundary}$ is a partial order.

  \medskip
  We obtain from Theorem~\ref{thm-box-bound} the following 
chain of implications for tableaux $X$, $Z$ of the same type such that the
box relation is defined:
$$Z\leq_{\sf box}X \quad\text{implies}\quad Z\leq_{\sf boundary}X \quad
\text{implies}\quad Z\leq_{\sf dom}X$$

  \medskip
  As a consequence, Theorem~\ref{thm-main} yields the following result:

  \begin{thm}
    The following statements are equivalent for LR-tableaux 
    $X$, $Z$ of the same type and of shape which is a rook strip.
    \begin{enumerate}
    \item $Z\leq_{\sf box} X$
    \item $Z\leq_{\sf boundary}X$
    \item $Z\leq_{\sf dom}X$\qed
    \end{enumerate}
  \end{thm}

  The case where the partition $\alpha$ has at most two parts
  is particularly well understood
  since then for each shape $\beta\setminus\gamma$,
  there are only finitely many isomorphism classes
  of short exact sequences in $\mathbb V_{\alpha,\gamma}^\beta$.
  In this situation, the boundary relation has a combinatorial description
  in terms of arc diagrams, see \cite[Theorem 1.2]{ks-hall},
  and is, in fact, transitive and equivalent
  to several algebraic and combinatorial relations, in particular to
  $\leq_{\sf box}$ and $\leq_{\sf dom}$.

%==============================================================
\section{The Bruhat order and the first proof of the main result} 
\label{section-first-proof}
%==============================================================

 \begin{notation}
   Let $S_n$ denote the symmetric group on $n$ letters.
   For any $i=1,\ldots,n-1$, denote by $s_i=(i,i+1)\in S_n$ the adjacent transposition.
 Let $x\in S_n$ and let $x=s_{i_1}s_{i_2}\cdots s_{i_k}$
 with $k$ minimal. Such an expression for $x$ is called
 {\it reduced}. The number $k$ is called the {\it length}
 of $x$ and we denote it by $\ell(x)=k$.

%For $x \in S_n$, define the {\it inversions} of $x$ to be pairs $a < b$ 
%such that $\pi(a) > \pi(b)$. 
\end{notation}

A key ingredient for our first proof of the main result is the notion
of Bruhat order on the symmetric group. An excellent reference for this
topic is \cite{bb}.
There are several ways to
define Bruhat order on $S_n$. We  define the Bruhat graph on $S_n$ by putting an edge from $u$ to $v$ if  
  $\ell(u)< \ell(v)$ and $ut=v$ for $t$ a transposition; we then define $u\leq v$ if there is a path from $u$ to $v$ in the Bruhat graph.
It follows directly from this definition that $u<v$ if and only if
$u^{-1}<v^{-1}$, using the fact that
  if $v=ut$, then $v^{-1}=u^{-1}(utu^{-1})$, where $utu^{-1}$ is also a transposition.
  We call $v$ {\it a~cover} of $u$ (and write $u\lessdot v$) if
  $u<v$ and there is no $w$ such that $u<w<v$.

The following well-known lemma characterizes the cover relations in
Bruhat order.

\begin{lem}[{\cite[Lemma 2.1.4, Theorem 2.2.6]{bb}}] \label{cover-lem}
  For $u,v$ in $S_n$, the following conditions are equivalent:
  \begin{itemize} \item $u\lessdot v$,
  \item $\ell(v)=\ell(u)+1$ and $v=ut$ for $t$ some transposition,
  \item $v=u(a,b)$, where $u(a)<u(b)$ and for $a<i<b$, we do not have
    $u(a)<u(i)<u(b)$.\end{itemize}\end{lem}

We will also need two other characterizations of Bruhat order.

\begin{thm}[{\cite[Theorem 2.2.2]{bb}}]\label{subword-th}
  For $u,v$ in $S_n$, let
  $s_{i_1}\dots s_{i_p}$ be a reduced expression for $v$. Then $u\leq v$ if
  and only if there is a subword of $s_{i_1}\dots s_{i_p}$ which is a reduced
  expression for $u$.\end{thm}

For a sequence of integers $(a_1,\dots,a_r)$, we write ${\sort}(a_1,\dots,a_r)$
for the sequence sorted into increasing order.

\begin{thm}[{\cite[Theorem 2.6.3]{bb}}] \label{bb-th}
  For $u,v\in S_n$, we have $u<v$ if and only if for all $1\leq i \leq n$,
  we have ${\sort}(u(1),\dots,u(i))$ has each entry less than or equal to
  the corresponding letter in ${\sort}(v(1),\dots,v(i))$.
\end{thm}

  Let $\alpha=(\alpha_1,\alpha_2,\dots)$ be a partition of $n$.  It will be convenient to fix once and
  for all a numbering of the boxes of $\alpha$: we number them by rows, from
  left to right, starting with the top row.  
  
\begin{notation}
  Let $T$ be a filling of $\alpha$ by the numbers 1 to $n$, with each number
  appearing once.  Define
  $\pi(T)$ to be the permutation obtained by reading
  $T$ by rows from left to right, starting from the top row, that is 
  $\pi(T)(i)$ is the number contained in box $i$.  
\end{notation}

We will be particularly interested in applying $\pi$ to SYTs, but we define
it more generally for convenience.

\begin{ex}
  Let $\alpha=(2,2,1)$.  Consider the following SYTs of shape $\alpha$:
$$ Z:\quad  
\begin{picture}(5,8)(0,3)
%\multiput(0,12)(3,0)4{\smbox}
\put(0,7){\alphambox{1}}
\put(3,7){\alphambox{4}}
\put(0,4){\alphambox{2}}
\put(3,4){\alphambox{5}}
\put(0,1){\alphambox{3}}
%\multiput(0,9)(3,0)3{\smbox}
\end{picture}
\quad, \;\;\;
X: \quad
 \begin{picture}(5,8)(0,3)
%\multiput(0,12)(3,0)4{\smbox}
\put(0,7){\alphambox{1}}
\put(3,7){\alphambox{2}}
\put(0,4){\alphambox{3}}
\put(3,4){\alphambox{4}}
\put(0,1){\alphambox{5}}
%\multiput(0,9)(3,0)3{\smbox}
 \end{picture} $$

Note that $\pi(Z)=(1,4,2,5,3)$ and $\pi(X)=(1,2,3,4,5)$.
 \end{ex}

\begin{prop}\label{prop-bruhat-dom} 
Let $X,Z$ be standard Young tableaux of shape $\alpha$.  The relation $Z \leq_{\sf dom} X$ holds 
if and only if $\pi(X) \leq\pi(Z )$ in the Bruhat order.
\end{prop}

\begin{sloppypar}
\begin{proof}
 Let the sequence of partitions corresponding to $X$ be $\gamma^{(1)},\dots,\gamma^{(|\alpha|)}$, and let the sequence of partitions corresponding to
    $Z$ be $\delta^{(1)},\dots,\delta^{(|\alpha|)}$.
   Since $\pi(X)\leq \pi(Z)$ if and only if $\pi(X)^{-1} \leq \pi(Z)^{-1}$, we can
   apply the criterion of Theorem \ref{bb-th} to $\pi(X)^{-1}$ and
   $\pi(Z)^{-1}$.
    The sequence 
${\sort}(\pi(X)^{-1}(1),\dots,\pi(X)^{-1}(i))$     
 consists
of a sorted list of the positions of $1,\dots,i$ in $X$, using the numbering
of the boxes of $\alpha$ defined previously.  The condition that
${\sort}(\pi(X)^{-1}(1),\dots,\pi(X)^{-1}(i))$ have each letter less than or equal to
the corresponding letter in ${\sort}(\pi(Z)^{-1}(1),\dots,\pi(Z)^{-1}(i))$ is equivalent
to the condition that $\gamma^{(i)}\geq_{\sf dom} \delta^{(i)}$;
all these conditions put together yield exactly the condition that
$X\geq_{\sf dom}Z$. 
\end{proof}
\end{sloppypar}

Now we consider the ``box move''. We recall that a~decreasing
box move on $X$ swaps two
entries of $X$, so that the larger entry moves higher,
and the smaller entry moves lower.
We have already established, in Lemma \ref{lem-box-to-dom},
that applying a~decreasing box move moves us down in dominance order. 
We wish to prove the converse result, that if $X > Z$ in dominance
order, then there is a~sequence of decreasing box moves which takes us
from $X$ to
$Z$.
Our strategy is inductive: we find a decreasing box move from $X$ to $Y$,
such that $Y \geq Z$ in dominance order.
In order to do this, we recast the problem in terms of Bruhat order on
permutations.

% {\red
%  \begin{lem}\label{lemma3}
%    For $x \in S_n$, we have that $x \in L$ if and only if the
%    $\alpha$-recording tableau for $x$ is standard.
%  \end{lem}
% 
% \begin{proof}
% The row-increasing condition for standardness is equivalent to $x \in S^\alpha$. Suppose this is satisfied, 
% and let $X$ be the corresponding tableau. The  number of entries $\leq t$ in
% row $i$ of the recording tableau is equal to the number of boxes in $X$ labeled $i$ we have seen
% in reading the first $t$ entries of $X$. Denote this number by $\lambda^{(t)}_i$. 
% The lattice condition is then equivalent to the
% condition that for each $t$, the resulting shape $(\lambda^{(t)}_1,\lambda^{(t)}_2,\ldots)$
% is a~partition, and this condition is
% easily seen to be equivalent to column-increasingness.
% \end{proof}
% 
% 
% 
% We can also think of a~standard $\alpha$-recording tableau as the record of a~process
% of adding a~sequence of single boxes, resulting in the final shape $\alpha$, such that at
% each step, the shape is a~partition.}

\begin{lem}\label{lemma4}
If $\pi(Y)$ covers $\pi(X)$ in Bruhat order, 
  then there is a decreasing box move from $X$ to $Y$.
  
\end{lem}

\begin{proof}
A~cover in Bruhat order swaps two entries, and moves the larger number up.
\end{proof}

Therefore, our desired result follows once we prove the following
result on permutations. Let
us write  $L\subseteq S_n$ for the set of permutations of the form $\pi(X)$ for
   $X$ a SYT of shape $\alpha$.

\begin{prop}\label{prop-bruhat2}
If $x < z$ in Bruhat order, with $x, z \in L$, then there is a~cover $x\lessdot y$
with $y \in L$ such that $y\leq z$.
\end{prop}

\begin{proof}
Write a reduced expression for $z$ as a~product of adjacent transpositions $s_{i_1}\cdots s_{i_p}$.
Since $x < z$, by Theorem \ref{subword-th}, there is a~subword of this word which equals $x$. As in the proof of
\cite[Lemma 2.2.1]{bb}, choose one such that the rightmost omitted transposition is as far
to the left as possible 
(i.e. $x=s_{i_1}\cdots \widehat{s}_{i_{j_1}}\cdots \widehat{s}_{i_{j_q}}\cdots s_{i_p}$
with $j_1<\cdots <j_q$ such that $j_q$ is minimal, 
and $\widehat{s}$ means that $s$ is omitted).  
Define $y$ to be obtained from this subword by adding back in
the rightmost transposition in the word for $z$ not in the chosen subword for $x$
(i.e., $y=s_{i_1}\cdots \widehat{s}_{i_{j_{1}}}\cdots \widehat{s}_{i_{j_{q-1}}}\cdots
s_{i_{j_q}}\cdots s_{i_p}$). 
By the proof of \cite[Lemma 2.2.1]{bb}, this is a~reduced expression for
a~permutation
which covers $x$ in Bruhat order and lies below $z$.

We will now show that $y\in L$.
We can write $y=xt$, where $t$ is the transposition
$s_{i_p}s_{i_{p-1}}\dots s_{i_{j_q}} \dots s_{i_{p-1}}s_{i_p}$ (i.e.,
the product of the simple transpositions starting at the righthand end of our expression
for $z$, proceeding backwards as far as $s_{i_{j_q}}$, and then proceeding forwards
again to the end).
From this description of $t$, it is clear that
$\ell(zt)<\ell(z)$, so $zt<z$ in Bruhat order.

%We now summarize the facts that were given or that we have established.
%The rest of the proof
%will use only them, and nothing further about the specific choice of $y$ that
%we made above will be needed:
%\begin{itemize}
%\item $x$ and $z$ are in $L$, $t$ is a transposition.
%\item $y=xt$, and $y> x$ is a cover in Bruhat order.
%\item $z>zt$.
%\end{itemize}

Let $X$ and $Z$ denote the SYTs corresponding to $x$ and $z$ respectively.
Let the transposition $t$ which we determined earlier be $(a, b)$ with $a< b$.
Multiplying on the right by $t$ swaps the entries in positions
$a$ and $b$. Since the effect of this on $z$ moves us down in Bruhat order, the
larger entry must be in position $a$, and the smaller in position $b$. This
implies that the boxes numbered $a$ and $b$ must be in distinct rows of
the diagram for $\alpha$, 
since the entries of $Z$ within
any single row are increasing. The same argument implies that $a$ and $b$
lie in distinct columns. 

Now consider the effect of swapping the entries in positions $a$ and $b$ on
$x$. We know that this yields $y=xt$ which covers $x$ in Bruhat order.
Thus, the entry of $x$ in position $a$ is less than that in position $b$.
Further, by Lemma \ref{cover-lem}, there are no entries with values between $x(a)$ and $x(b)$ and
located between positions $a$ and $b$. This implies that swapping the entries
in positions $a$ and $b$ of $X$ will result in a standard tableau, so $y\in L$,
and we are done.
\end{proof}

\begin{rem}
This proof for the implication 
$X>_{\sf dom}Z\implies X>_{\sf box}Z$
is constructive as it exhibits the first box move:
Let $s_{i_1}\cdots s_{i_p}$ be a reduced expression for $\pi(Z)$.
Write $\pi(X)$ as a subword $s_{i_1}\cdots \hat s_{i_{j_1}}\cdots
\hat s_{i_{j_q}}\cdots s_{i_p}$ such that $j_q$ is minimal.
Then $s_{i_p}\cdots s_{i_{j_q}}\cdots s_{i_p}=(a,b)$
is a transposition and $\pi(X)(a, b)=\pi(Y)$ defines a~standard
  Young tableau $Y$ which satisfies $Z\leq_{\sf dom}Y$ and 
$Y<_{\sf box}X$.
\end{rem}

\begin{ex}
Consider:
  $$ Z:\quad  
\begin{picture}(5,8)(0,3)
%\multiput(0,12)(3,0)4{\smbox}
\put(0,7){\alphambox{1}}
\put(3,7){\alphambox{4}}
\put(0,4){\alphambox{2}}
\put(3,4){\alphambox{5}}
\put(0,1){\alphambox{3}}
%\multiput(0,9)(3,0)3{\smbox}
\end{picture}
\quad, \;\;\;
X: \quad
 \begin{picture}(5,8)(0,3)
%\multiput(0,12)(3,0)4{\smbox}
\put(0,7){\alphambox{1}}
\put(3,7){\alphambox{2}}
\put(0,4){\alphambox{3}}
\put(3,4){\alphambox{4}}
\put(0,1){\alphambox{5}}
%\multiput(0,9)(3,0)3{\smbox}
 \end{picture} $$

We have
$\pi(Z)=14253$, and $\pi(X)=12345$.  Let's say we want to move down from $X$.  
We write $\pi(Z)= (34)(23)(45)=s_3s_2s_4$.  
Now $\pi(X)=e$, so the subword of $s_3s_2s_4$ corresponding to $e$ is
the empty subword.  We add the rightmost reflection back in, so that is 
$(a, b)=s_4$.  
This gives us the permutation $y_1=12354$.   To get the next step down  the chain, we go to 
the subword of $z$ given by $y_1=s_4=12354$.
 Here  we right multiply $y_1$ by $s_4s_2s_4=s_2$, obtaining $13254$.
In the final step we right multiply $y_2$ by $s_4s_2s_3s_2s_4=(2,5)$, obtaining $\pi(Z)$.

The corresponding sequence of fillings is:
  $$
   \begin{picture}(5,8)(0,3)
%\multiput(0,12)(3,0)4{\smbox}
\put(0,7){\alphambox{1}}
\put(3,7){\alphambox{2}}
\put(0,4){\alphambox{3}}
\put(3,4){\alphambox{4}}
\put(0,1){\alphambox{5}}
%\multiput(0,9)(3,0)3{\smbox}
   \end{picture} \quad \raisebox{1.5mm}{$\gtrdot$} \quad
    \begin{picture}(5,8)(0,3)
%\multiput(0,12)(3,0)4{\smbox}
\put(0,7){\alphambox{1}}
\put(3,7){\alphambox{2}}
\put(0,4){\alphambox{3}}
\put(3,4){\alphambox{5}}
\put(0,1){\alphambox{4}}
%\multiput(0,9)(3,0)3{\smbox}
    \end{picture} \quad \raisebox{1.5mm}{$\gtrdot$} \quad
     \begin{picture}(5,8)(0,3)
%\multiput(0,12)(3,0)4{\smbox}
\put(0,7){\alphambox{1}}
\put(3,7){\alphambox{3}}
\put(0,4){\alphambox{2}}
\put(3,4){\alphambox{5}}
\put(0,1){\alphambox{4}}
%\multiput(0,9)(3,0)3{\smbox}
     \end{picture} \quad \raisebox{1.5mm}{$\gtrdot$} \quad
      \begin{picture}(5,8)(0,3)
%\multiput(0,12)(3,0)4{\smbox}
\put(0,7){\alphambox{1}}
\put(3,7){\alphambox{4}}
\put(0,4){\alphambox{2}}
\put(3,4){\alphambox{5}}
\put(0,1){\alphambox{3}}
%\multiput(0,9)(3,0)3{\smbox}
 \end{picture}$$

\end{ex}

%========================================================
\section{The second proof of the main result} \label{section-algorithm}
%========================================================

Let $X$, $Z$ be SYTs of the same shape.
In order to show that
the dominance relation implies the box relation
we present here a simple and explicit procedure
to determine box moves on SYTs which transform  $X$  into  $Z$.
If $Z<_{\sf dom}X$, the following steps yield an SYT $Y$ such that
$Z\leq_{\sf dom}Y$ and $Y<_{\sf box} X$.

\begin{itemize}
  \item[(1)] Find the largest entry $k$ which is in different positions
    in $X$ and $Z$.
  \item[(2)] Let $a$ be the row of the entry $k$ in $Z$;
  \item[(3)] let $m$ be the entry in that position in $X$, note that $m<k$;
  \item [(4)] let $b>a$ be the first row below $a$ which contains an
    entry in the range $m<\cdot\leq k$ in $X$;
  \item[(5)] let $\ell$ be the smallest entry with $m<\ell\leq k$
    in row $b$ of $X$.
  \item[(6)] The SYT $Y$ is obtained from $X$ by swapping the entries
    $\ell$ and $m$.
\end{itemize}

\begin{ex}
  The following SYTs $X$ and $Z$ have shape $(3,2,1)$ and are in dominance order:
   $$
  \raisebox{1.5mm}{$Z:$} \quad \alphatab 126345 \;\;\;\;\quad \raisebox{1.5mm}{$<_{\sf dom}$} \quad \;\;
  \raisebox{1.5mm}{$X:$} \quad \alphatab 123465
  $$
  In the algorithm we compute:  $k=6$, $a=1$, $m=3$, $b=2$ and $\ell=4$. Then
  $$
  \raisebox{1.5mm}{$Y:$} \quad \alphatab124365
  $$
  satisfies $Z<_{\sf dom} Y$ and we can continue.
\end{ex}

  We illustrate this algorithm using tables which
  describe the difference of the two SYTs.
    Let $X$ and $Z$ be given by partition sequences
    $[\gamma^{(1)},\ldots,\gamma^{(|\alpha|)}]$ and $[\delta^{(1)},\ldots,\delta^{(|\alpha|)}]$,
    respectively.  
  Each column of the table $T_{X-Z}$ is indexed by
  the partition number $j$, each row by the row number $i$.
  The entries are simply $(T_{X-Z})_{i,j}=\gamma_i^{(j)}-\delta_i^{(j)}$.

  In the above example, $X$ and $Z$ have partition sequences
  \begin{eqnarray*}X & : & [(1),(2),(3),(3,1),(3,1,1),(3,2,1)],\\ 
    Z & : & [(1),(2), (2,1),(2,2), (2,2,1),(3,2,1)],
  \end{eqnarray*}
  respectively.
  Here is the corresponding table.

  $$
  \begin{array}{c|c;{1pt/2pt}c;{1pt/2pt}c;{1pt/2pt}c;{1pt/2pt}c;{1pt/2pt}c}
    T_{X-Z} & \;1\; & \;2\; & \;3\; & \;4\; & \;5\; & \;6\;  \\ \hline
    1 & 0 & 0 & 1 & 1 & 1 & 0 \\ \hdashline[1pt/2pt]
    2 & 0 & 0 & -1 & -1 & -1 & 0  \\ \hdashline[1pt/2pt]
    3 & 0 & 0 & 0 & 0 & 0 & 0 
  \end{array}
  $$
  The table $T_{Y-Z}$ for the tableaux after the box move differs from $T_{X-Z}$
  in column 3 which is zero.

  \medskip
  The following lemma is obtained as a~consequence of the definition of the dominance order.

  \begin{lem}
    \label{lemma-table}
    Two SYTs $X$ and $Z$ are in dominance order, $Z\leq_{\sf dom}X$,
    if and only if
    for each column $j$ and each row $a$ in the table $T_{X-Z}$,
    the sum $\sum_{i=1}^a(T_{X-Z})_{i,j}$ of all entries in the $j$-th
    column from the top down to row $a$ is nonnegative. \qed
  \end{lem}
  
  \medskip
  We can now give the second proof of Theorem~\ref{thm-main}.

  \begin{proof}
    For two SYTs $X$, $Z$ of the same shape $\alpha$ we need to show that they are in
    dominance order if and only if they are in box order.  The ``only if'' part
    has been shown in Lemma~\ref{lem-box-to-dom}.  For the converse, assume that
    $X$ and $Z$ are in dominance order, say $Z\leq_{\sf dom} X$.  If $X$ and $Z$ are
    equal in dominace order, then $X=Z$ and we are done.

    \smallskip
    We show that in case $Z<_{\sf dom} X$, the steps described above
    will produce a SYT $Y$ such that $Z\leq_{\sf dom} Y$ and $Y<_{\sf box} X$.

    \smallskip
    Consider the table $T_{X-Z}$, we discuss the entries as steps (1) through (4)
    are being performed.
    
    $$
    \begin{array}{c|c;{1pt/2pt}c;{1pt/2pt}c;{1pt/2pt}c;{1pt/2pt}c;{1pt/2pt}c;{1pt/2pt}c;{1pt/2pt}c;{1pt/2pt}c}
      T_{X-Z} & \cdots & m & \cdots & \ell-1 & \ell & \cdots & k-1     &\; k \; & \cdots   \\ \hline
      \vdots & & \phantom{xx}   & &        & \phantom{xx}     & &  0      & 0  & \cdots \\ \hdashline[1pt/2pt]
      a   & & > & > & \cdots  & \cdots & >       & 1 & 0 & \cdots \\ \hdashline[1pt/2pt]
      \vdots & & \geq & \geq & \cdots & \cdots  & \geq & 0 & 0 & \cdots \\ \hdashline[1pt/2pt]
      b   & & & & & & & 0 & 0 & \cdots \\ \hdashline[1pt/2pt]
      \vdots & & & & & & & 0 & 0 & \cdots \\ \hdashline[1pt/2pt]
      a'  & & & &  & & & -1 & 0 & \cdots \\ \hdashline[1pt/2pt]
      \vdots & & & & & & & 0 & 0 & \cdots
    \end{array}
    $$

    \noindent{\bf (1)} Let $k$ be the largest entry which is in different positions in $X$ and $Z$.

    \smallskip
    By the choice of $k$, the partitions $\gamma^{(i)}$, $\delta^{(i)}$ are equal for $i\geq k$.
    Hence in the table, the $k$-th column, and everything on its right, is zero.

    \smallskip\noindent{\bf (2)} Let $a$ be the row in which $k$ occurs in $Z$.

    \smallskip
    The entry $k$ occurs in different rows in $X$ and $Z$, say it occurs 
    in row $a'$ in $X$.  Recall that $\delta^{(k)}=\gamma^{(k)}$, hence $\delta^{(k-1)}$ and $\gamma^{(k-1)}$
    differ only in the two rows $a$ and $a'$ where entry $k$ occurs in $Z$ and $X$, respectively.
    Then $\delta^{(k-1)}_i=\gamma^{(k-1)}_i$ for $i\neq a,a'$ and
    $$\delta^{(k-1)}_a=\gamma^{(k-1)}_a-1,\quad \delta ^{(k-1)}_{a'}=\gamma^{(k-1)}_{a'}+1.$$
    Since $Z<_{\sf dom}X$ are in dominance relation, we have $\delta^{(k-1)}<_{\sf dom}\gamma^{(k-1)}$, hence $a'>a$.

    \smallskip\noindent{\bf (3)} Let $m$ be the entry in $X$ which has the position of $k$ in $Z$.

    \smallskip
    Note that $m<k$ since all entries greater than $k$ are in the same position in $X$ and $Z$.
    Considering the $a$-th rows of the two tableaux, it follows that
    $$\gamma^{(i)}_a>\delta^{(i)}_a \quad \text{for}\quad m\leq i<k.$$
    
    \smallskip\noindent{\bf (4)} Let $b>a$ be the first row below $a$ which contains an
    entry in the range $m<\cdot\leq k$ in $X$.

    \smallskip
    Such a row $b$ exists since $a'>a$ and row $a'$ contains the entry $k$.
    Consider a row $r$ in the table $T_{X-Z}$ with $a<r<b$.  The entries greater than $k$
    are in the same positions in $X$ and in $Z$, and none of the entries
    in the range $m<\cdot\leq k$ occur in row $r$ in $X$,
    hence for each entry $s$ with $m<s<k$, $\gamma_r^{(s)}-\delta_r^{(s)}\geq 0$.
    Thus, the entries in $T_{X-Z}$ in the rectangle given by rows $m\leq\cdot<k-1$
    and columns $a<\cdot<b$ are all nonnegative.
    This finishes the proof that the given entries in the table $T_{X-Z}$ are as specified.
    
    \smallskip\noindent{\bf (5)} Let $\ell$ be the minimal entry in row $b$ of $X$
    with  $m<\ell\leq k$.

    \smallskip
    We claim that the column in $X$ containing $\ell$ is on the left of 
    the column in $X$ containing $m$.
    For this, recall that $Z$ has $k$ in the position where $X$ has $m$.
    In the quadrant underneath and to the right of $k$ in the SYT $Z$,
    all entries are greater than or equal to $k$.
    Such entries are in the same position in $X$ and $Z$, which
    implies the claim.

    \smallskip\noindent{\bf (6)} To obtain $Y$ from $X$, exchange the entries $\ell$ and $m$.

    \smallskip
    We show that $Y$ is a~SYT. By the choice of $\ell$, the entries in  row $b$ of $Y$ are strictly
    increasing; by the choice of $k$ and $m$, the entries in row $a$ of $Y$ are strictly increasing.

    Now, consider the column of $Y$ that contains the entry $m$. Since $m<\ell$ and $X$ is a~SYT,
    all entries in rows of higher numbers are bigger than $m$.
    Assume that there exists an entry $n>m$ in the column of $m$ and above $m$.
    By the choice of $b$ (minimality),
    it follows that $n$ is in row $a$ or above.  This is a~contradiction, because $X$ is a~SYT. 

    Regarding the column of $Y$ that contains the entry $\ell$, recall that 
    this entry is in the position where $X$ has $m$ and $Z$ has $k$.
    Since both $X$ and $Z$ are SYTs, the entries in this column in $Y$ must be strictly increasing.
    This finishes the proof that $Y$ is a~SYT, and, using the claim under (5),
    that $Y$ is obtained from $X$ by a decreasing box move.

    It remains to prove that $Z\leq_{\sf dom} Y$.  For this, we first verify that the
    table $T_{X-Y}$ for the tableaux $X$ and $Y$ has only zero entries, except in the positions
    indicated.  This is straightforward since $Y$ is obtained from $X$ by a box move which
    exchanges entries $\ell$ and $m$ in rows $a$ and $b$.
    $$
    \begin{array}{c|c;{1pt/2pt}c;{1pt/2pt}c;{1pt/2pt}c;{1pt/2pt}c;{1pt/2pt}c}
      T_{X-Y} & \cdots & m & \cdots & \ell-1 & \ell & \cdots \\ \hline
      \vdots & & \phantom{xx}   & &        & \phantom{xx}     &  \\ \hdashline[1pt/2pt]
      a   & & 1 & \cdots & 1 &  &        \\ \hdashline[1pt/2pt]
      \vdots & & & & & &   \\ \hdashline[1pt/2pt]
      b   & & -1 & \cdots & -1  & & \\ \hdashline[1pt/2pt]
      \vdots & & & & & \phantom 0 & 
    \end{array}
    $$
    The table $T_{Y-Z}$ is obtained by subtracting the entries in $T_{X-Y}$
    from those in $T_{X-Z}$.
    Since $X$ is a~SYT, the columns for $T_{X-Z}$ have nonnegative partial sums, as specified in
    Lemma~\ref{lemma-table}.
    Using the particular form of $T_{X-Z}$ established in steps (1) through (4), it follows that also the
    columns for $T_{Y-Z}$ have nonnegative partial sums.  Applying
    Lemma~\ref{lemma-table} again yields $Y\geq_{\sf dom} Z$.

    \smallskip
    We have seen that $Y$ is obtained from $X$ by a decreasing box move.
    By repeating this process, starting with the SYTs $Y$ and $Z$, we produce in finitely many
    rounds the desired sequence of box moves.
  \end{proof}

%===================================================================
\section{Combinatorial properties of the order $\leq_{\sf box}$}
%===================================================================
\label{sec-prop}

In this section we study combinatorial properties of the
poset $(\mathcal T_\alpha,\leq_{\sf box})$,
  and discuss in Section~\ref{section-applic} an application to invariant subspaces
  of nilpotent linear operators.

\subsection{ An example and some properties of $\leq_{\sf box}$}

Here is the poset $(\mathcal T_{(3,2,1)},\leq_{\sf box})$.

$$
\beginpicture\setcoordinatesystem units <1cm,1.5cm>
\put{\alphatab 123456} at 10 6
\put{\alphatab 124365} at 10 4
\put{\alphatab 134265} at 10 3
\put{\alphatab 135264} at 10 2
\put{\alphatab 146253} at 10 0
\put{\alphatab 145263} at 12 1
\put{\alphatab 125364} at 12 3
%\putrectangle corners at 11.7 4.5 and 12.7 5.3
%\putrectangle corners at 3.7 2.5 and 4.7 3.3
\put{\alphatab 123465} at 12 5
\put{\alphatab 136254} at 8 1
\put{\alphatab 126354} at 8 2
\put{\alphatab 134256} at 8 4
\put{\alphatab 124356} at 8 5
\put{\alphatab 125346} at 6 4
\put{\alphatab 135246} at 6 3
\put{\alphatab 136245} at 6 2
\put{\alphatab 126345} at 4 3
\arr{10 2.3}{10 2.5}
\arr{10 3.3}{10 3.5}
\arr{8 4.3}{8 4.5}
\arr{8 1.3}{8 1.5}
\arr{6 2.3}{6 2.5}
\arr{6 3.3}{6 3.5}
\arr{10.8 0.3}{11.5 0.65}
\arr{11.5 1.25}{10.7 1.7}
\arr{10.8 2.3}{11.5 2.65}
\arr{11.5 3.25}{10.7 3.7}
\arr{10.8 4.3}{11.5 4.65}
\arr{11.5 5.25}{10.7 5.7}
\arr{8.8 1.3}{9.5 1.65}
\arr{9.5 0.25}{8.7 0.7}
\arr{8.8 5.3}{9.5 5.65}
\arr{9.5 4.25}{8.7 4.7}
\arr{9.5 3.25}{8.7 3.7}
\arr{7.5 1.25}{6.7 1.7}
\arr{6.8 4.3}{7.5 4.65}
\arr{6.8 3.3}{7.5 3.65}
\arr{11.3 3.17}{6.6 3.67}

\arr{8.8 2.1}{11.4 2.75}
\arr{6.8 2.1}{9.4 2.75}
\arr{7.4 2.1}{4.9 2.7}
\arr{9.4 2.1}{6.8 2.75}
\arr{5.5 2.25}{4.8 2.65}
\arr{4.8 3.3}{5.5 3.65}
\arr{4.9 3.2}{9.5 3.7}
\endpicture
$$

  We observe:
  \begin{itemize}
  \item in this poset there exists exactly one maximal and exactly one minimal element;
  \item all saturated chains have the same length;
  \item this poset is not a~lattice.
  \end{itemize}

  The first two properties hold for each
poset of the form  $(\mathcal T_\alpha, \leq_{\sf box})$.
First, we verify that the poset has a unique maximal and a unique
minimal element.
  
%    We will now verify that the first two properties extend
%     to all $\mathcal (T_{\alpha},\leq_{\sf box})$.
%---------------------------------------------------------------
%\subsection{Maximal and minimal elements} 
%---------------------------------------------------------------
%
\label{ex-vertical-strip-necessary}
%
%   First, we verify that for each partition $\alpha$,
%   the poset $(\mathcal T_{\alpha},\leq_{\sf box})$ has a unique maximal and a unique
%   minimal element.

  \begin{lem}Let $\alpha$ be a partition.
    \label{lemma-minmax}
    \begin{enumerate}
    \item There is a unique maximal element $X$ in the poset $(\mathcal T_\alpha,\leq_{\sf box})$.
      The SYT for $X$ is given by writing the numbers $1,\ldots,|\alpha|$ in  numerical
      order into the Young diagram for $\alpha$.
    \item There is a unique minimal element $Z$ in the poset $(\mathcal T_\alpha,\leq_{\sf box})$.
      The SYT for $Z$ is the transpose of the Young diagram for $\alpha'$ in which the numbers
      $1,\ldots,|\alpha|$ are entered in  numerical order.
    \end{enumerate}
  \end{lem}

  \begin{proof}
    1. If two numbers $i, i+1$ do not occur in  numerical order in the SYT for $X$,
    then the box move exchanging $i$ with $i+1$ yields a SYT
    which is larger in the dominance order.

    2.  By definition, taking transposes is an order-reversing bijection between the posets
    $(\mathcal T_\alpha,\leq_{\sf box})$ and $(\mathcal T_{\alpha'},\leq_{\sf box})$.
  \end{proof}

%-------------------------------------------------------------------------------
% \subsection{Saturated chains}
%-------------------------------------------------------------------------------

Next, we verify that all saturated chains in these posets have the same length.
  
  \begin{prop}\sloppypar
    For a partition $\alpha$, all saturated chains in the poset
    $(\mathcal T_\alpha,\leq_{\sf box})$ have the same length.
  \end{prop}

\begin{proof}
  Since  covers in the Bruhat order increase
  the length by $1$ (Lemma \ref{cover-lem}), the result follows from
  Theorem \ref{thm-main}, Proposition \ref{prop-bruhat-dom}, Lemma \ref{lemma4}
  and Proposition \ref{prop-bruhat2}.
\end{proof}

%--------------------------------------------------------------------------------
\subsection{An application to invariant subspaces}
%--------------------------------------------------------------------------------
\label{section-applic}

We return to the link between LR-tableaux and varieties of invariant subspaces of
nilpotent linear operators studied in Section~\ref{sec-motivation}.

\medskip
It follows from Lemma~\ref{lemma-minmax} that the poset  $(\mathcal T_{\alpha,\gamma}^\beta,
\leq_{\sf dom})$ has a unique minimal and a unique maximal element provided the
partitions $\alpha$, $\beta$, $\gamma$ are such that $\beta\setminus\gamma$ is a rook strip.
This extends the corresponding result \cite[Proposition~5.5]{ks-survey} for
the case where $\beta$ and $\gamma$ are arbitrary but all parts of $\alpha$ are at most~2.  

\medskip
In terms of varieties of invariant subspaces, 
Lemma~\ref{lemma-minmax} has the following interpretation.

\begin{cor}
  Given partitions $\alpha$, $\beta$, $\gamma$ such that $\beta\setminus\gamma$
  is a rook strip, the variety
  $$\mathbb V_{\alpha,\gamma}^\beta=\bigcup_{X\in\mathcal T_{\alpha,\gamma}^\beta}\mathbb V_X$$
  has a unique open component $\mathbb V_O$ and a unique closed
  component $\mathbb V_C$.
  The open component $\mathbb V_O$ and the closed component $\mathbb V_C$
  are given by the LR-tableaux $O$ and $C$, respectively,
  which correspond to the
  unique minimal and the unique maximal element in the poset
  $(\mathcal T_{\alpha,\gamma}^\beta,\leq_{\sf dom})$.
\end{cor}

\begin{proof}
  The following statements are equivalent for an LR-tableau $O\in
  \mathcal T_{\alpha,\gamma}^\beta$:
  (1)  $O$ is a minimal element in the poset
      $(\mathcal T_{\alpha,\gamma}^\beta,\leq_{\sf dom})$;
  (2) $O$ is a minimal element in the poset
      $(\mathcal T_{\alpha,\gamma}^\beta,\leq_{\sf boundary})$;
  (3)  there is no $X\in\mathcal T_{\alpha,\gamma}^\beta\setminus\{O\}$
      such that $\mathbb V_O\cap \overline{\mathbb V}_X\neq \emptyset$;
  (4) the union $\cup_{X\neq O}\mathbb V_X$ is closed in
      $\mathbb V_{\alpha,\gamma}^\beta$;
  (5) the variety $\mathbb V_O$ is open in $\mathbb V_{\alpha,\gamma}^\beta$.
      Hence, since the poset $(\mathcal T_{\alpha,\gamma}^\beta,\leq_{\sf dom})$
      has $O$ as the unique minimal element, the component
      $\mathbb V_O$ is the unique open component
      in the decomposition in the lemma.

      Similarly, the following statements are equivalent for $C\in 
  \mathcal T_{\alpha,\gamma}^\beta$:
  (1)  $C$ is a maximal element in the poset
      $(\mathcal T_{\alpha,\gamma}^\beta,\leq_{\sf dom})$;
  (2) $C$ is a maximal element in the poset
      $(\mathcal T_{\alpha,\gamma}^\beta,\leq_{\sf boundary})$;
  (3)  there is no $X\in\mathcal T_{\alpha,\gamma}^\beta\setminus\{C\}$
      such that $\mathbb V_X\cap \overline{\mathbb V}_C\neq \emptyset$;
  (4) the variety $\mathbb V_C$ is closed in $\mathbb V_{\alpha,\gamma}^\beta$.
     The uniqueness of the maximal element $C$ in 
     $(\mathcal T_{\alpha,\gamma}^\beta,\leq_{\sf dom})$ implies that $\mathbb V_C$
     is the unique closed component in the decomposition.
\end{proof}

 We now consider some examples in which we relax the condition
  that $\beta\setminus \gamma$ is a rook strip. We note that in this case,
  the definition of box moves is modified to allow resorting the
  rows (if $\beta\setminus\gamma$ is a horizontal strip) or the columns
  (if $\beta\setminus\gamma$ is a vertical strip).
%   , though this does not affect the examples below, {\blue see \cite[page 959]{ks-strips}}. {\color{red} ADD A REFERENCE FOR THIS?}

 \begin{ex} 
 \begin{enumerate} 
            
            \item 
              The first example shows that the condition that $\beta\setminus\gamma$ be a~horizontal strip is needed for the uniqueness of the maximal element in
                $\mathcal T_{\alpha,\gamma}^\beta$.
                Consider the partition triple $\beta=(4,3,2,2,1)$, $\gamma=(3,2,1,1)$ and $\alpha=(2,2,1)$.
                The Hasse diagram of the poset $(\mathcal{T}^\beta_{\alpha,\gamma},\leq_{\sf dom})$
                 (or equivalently of $(\mathcal{T}^\beta_{\alpha,\gamma},\leq_{\sf box})$) has the following shape:
  
 $$
 
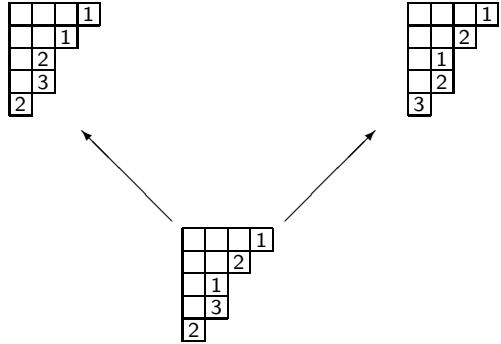
\begin{picture}(80,40)(0,-5)
   \put(7,30){\setcounter{boxsize}{3}
     \begin{picture}(18,12)(0,6)
       \multiput(0,12)(3,0)4{\smbox}
       \put(9,12){\alphambox{1}}
       \multiput(0,9)(3,0)3{\smbox}
       \put(6,9){\alphambox{1}}
       \multiput(0,6)(3,0)2{\smbox}
       \put(3,6){\alphambox{2}}
       \put(0,3){\smbox}
       \put(3,3){\alphambox{3}}
       \put(0,0){\alphambox{2}}
   \end{picture}}
   \put(60,30){\setcounter{boxsize}{3}
     \begin{picture}(18,12)(0,6)
       \multiput(0,12)(3,0)4{\smbox}
       \put(9,12){\alphambox{1}}
       \multiput(0,9)(3,0)3{\smbox}
       \put(6,9){\alphambox{2}}
       \multiput(0,6)(3,0)2{\smbox}
       \put(3,6){\alphambox{1}}
       \put(0,3){\smbox}
       \put(3,3){\alphambox{2}}
       \put(0,0){\alphambox{3}}
   \end{picture}}
   \put(30,0){\setcounter{boxsize}{3}
     \begin{picture}(18,12)(0,6)
       \multiput(0,12)(3,0)4{\smbox}
       \put(9,12){\alphambox{1}}
       \multiput(0,9)(3,0)3{\smbox}
       \put(6,9){\alphambox{2}}
       \multiput(0,6)(3,0)2{\smbox}
       \put(3,6){\alphambox{1}}
       \put(0,3){\smbox}
       \put(3,3){\alphambox{3}}
       \put(0,0){\alphambox{2}}
     \end{picture}
   }
   \put(45,10){\vector(1,1){12}}
   \put(30,10){\vector(-1,1){12}}
 \end{picture}
 $$ 
 
 \item 
     The second example shows that in Theorem \ref{thm-main},
     the condition that $\beta\setminus\gamma$ be a~vertical strip
     is necessary.  
     We also see that for horizontal strips, the poset $(\mathcal T_{\alpha,\gamma}^\beta,\leq_{\sf box})$ 
     may have several minimal and several maximal elements.

   Let $\beta=(5,4,3,1)$, $\gamma=(4,3,2,1)$ and $\alpha=(2,2,1)$. There are two
 LR-fillings of type $(\alpha,\beta,\gamma)$: 
 
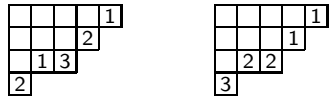
$$ \setcounter{boxsize}{3}
\begin{picture}(18,12)(0,2)
\multiput(0,12)(3,0)5{\smbox}
\put(12,12){\alphambox{1}}
\multiput(0,9)(3,0)4{\smbox}
\put(9,9){\alphambox{2}}
\multiput(0,6)(3,0)3{\smbox}
\put(3,6){\alphambox{1}}
\put(6,6){\alphambox{3}}
\put(0,3){\alphambox{2}}
\end{picture}
\;\;\;\;\;\;\;\;
\setcounter{boxsize}{3}
\begin{picture}(18,12)(0,2)
\multiput(0,12)(3,0)5{\smbox}
\put(12,12){\alphambox{1}}
\multiput(0,9)(3,0)4{\smbox}
\put(9,9){\alphambox{1}}
\multiput(0,6)(3,0)3{\smbox}
\put(3,6){\alphambox{2}}
\put(6,6){\alphambox{2}}
\put(0,3){\alphambox{3}}
\end{picture}
$$
 They are incomparable in $\leq_{\sf box}$ relation, but
$$ \setcounter{boxsize}{3}
\begin{picture}(18,12)(0,6)
\multiput(0,12)(3,0)5{\smbox}
\put(12,12){\alphambox{1}}
\multiput(0,9)(3,0)4{\smbox}
\put(9,9){\alphambox{2}}
\multiput(0,6)(3,0)3{\smbox}
\put(3,6){\alphambox{1}}
\put(6,6){\alphambox{3}}
\put(0,3){\alphambox{2}}
\end{picture}
\leq_{\sf dom}\;\;
\setcounter{boxsize}{3}
\begin{picture}(18,12)(0,6)
\multiput(0,12)(3,0)5{\smbox}
\put(12,12){\alphambox{1}}
\multiput(0,9)(3,0)4{\smbox}
\put(9,9){\alphambox{1}}
\multiput(0,6)(3,0)3{\smbox}
\put(3,6){\alphambox{2}}
\put(6,6){\alphambox{2}}
\put(0,3){\alphambox{3}}
\end{picture}
$$
            \end{enumerate}
 \end{ex}

%\begin{blue}
%============================================================================
\subsection{Open  Question} \label{section-conjecture}
%============================================================================

%\begin{red}
%  MS: I suggest to consider to rename the section as ``Open Question''
%  and make it a subsection (number 6.5).
%  Perhaps call the ``Conjecture'' below a ``Question''.
%\end{red}

We want to compare the two algorithms presented above.  Suppose we have
  two SYTs $X,Z$ of shape $\alpha$,
  such that $X >_{\sf dom} Z$.
    Say the first algorithm produces a SYT $\tilde Y$, the second a SYT
    $Y$, both of shape $\alpha$.  The tableaux satisfy
    $$X>_{\sf dom}\tilde Y\geq_{\sf dom}Z,\qquad X>_{\sf dom} Y\geq_{\sf dom}Z.$$

 The first  algorithm depends on
the choice 
of a factorization of $z$ as a product of adjacent transpositions.
 In this section, we describe 
a way to make the choice so that (conjecturally) the two algorithms
produce the same result, i.e., so that $Y=\tilde Y$.
Suppose that to produce the reduced expression for $z\in S_n$, a version of the ``bubble
sort'' algorithm is used, as in \cite[Example 3.4.3]{bb}.

Starting from the permutation $z$, we will find a sequence of adjacent
transpositions which transform $z$ into the identity permutation.
Start by moving $n$ into the final
position. If $n$ started in position $u$, this gives us the permutation
$z \cdot s_u \cdots s_{n-1} = z \cdot z_{(n)}^{-1}$. Next, move entry $n-1$ into its
proper position, which gives us $z \cdot z_{(n)}^{-1} \cdot s_{u'}\cdots s_{n-2} = z \cdot z_{(n)}^{-1} \cdot z_{(n-1)}^{-1}$, where $u'$ was the position of $n-1$ in
$z \cdot z_{(n)}^{-1}$. Continue with the entry $n-2$ and so forth. We
obtain $e=z \cdot z_{(n)}^{-1}\cdots z_{(2)}^{-1}$. In other words,
$z_{(2)}\cdots z_{(n)}$ is a reduced expression for $z$.

\begin{conj}
   If we apply the first algorithm with the  choice of reduced expression
   described above, its result agrees with the result of applying the second
   algorithm. 
 \end{conj}

%\end{blue}

  \subsection*{Acknowledgements}

  We are very grateful to the referees, whose suggestions improved the
  exposition of the paper, and significantly simplified the proof of
  Proposition \ref{prop-bruhat2}.

%\pagebreak[3]
%============================================================================

  \bigskip
Addresses of the authors:

\parbox[t]{5.5cm}{\footnotesize\begin{center}
              Faculty of Mathematics\\
              and Computer Science\\
              Nicolaus Copernicus University\\
              ul.\ Chopina 12/18\\
              87-100 Toru\'n, Poland\end{center}}
\parbox[t]{5.5cm}{\footnotesize\begin{center}
              Department of\\
              Mathematical Sciences\\ 
              Florida Atlantic University\\
              777 Glades Road\\
              Boca Raton, Florida 33431\end{center}}

\parbox[t]{5.5cm}{\centerline{\footnotesize\tt justus@mat.umk.pl}}
           \parbox[t]{5.5cm}{\centerline{\footnotesize\tt markus@math.fau.edu}}

\parbox[t]{5.5cm}{\footnotesize\begin{center}
              LaCIM\\
              D\'epartement de math\'ematiques\\ 
              UQAM\\
              C.P 8888, succursalle Centre-ville\\
              PK-5151\\
              Montr\'eal, QC H3C 3P8\end{center}}

\parbox[t]{5.5cm}{\centerline{\footnotesize\tt hugh.ross.thomas@gmail.com}}

\end{document}